\documentclass[12pt]{amsart}

\newtheorem{theorem}{Theorem}[section]
\newtheorem{lemma}[theorem]{Lemma}
\newtheorem{proposition}[theorem]{Proposition}
\newtheorem{corollary}[theorem]{Corollary}

\theoremstyle{definition}

\newtheorem{conjecture}[theorem]{Conjecture}

\theoremstyle{remark}
\newtheorem{remark}[theorem]{Remark}

\numberwithin{equation}{section}

\usepackage{enumerate}

\begin{document}

\title[Projective Hulls and Meromorphic Functions]{Projective Hulls and Characterizations of Meromorphic Functions}


\author[J.T. Anderson]{John T. Anderson}
\address{Department of Mathematics and Computer Science, College of the Holy Cross, Worcester, MA 01610, U.S.A.}
\email{anderson@mathcs.holycross.edu}
\author[J.A. Cima]{Joseph A. Cima}
\address{Department of Mathematics, University of North Carolina, Chapel Hill, NC 27599, U.S.A. }
\email{cima@email.unc.edu}

\author[N. Levenberg]{Norm Levenberg}
\address{Department of Mathematics, Indiana University, Bloomington, IN, 47405, U.S.A.}
\email{nlevenbe@indiana.edu}

\author[T.J. Ransford]{Thomas J. Ransford}
\address{D\'{e}partment de math\'{e}matiques et de statistique, Universit\'{e} Laval, Qu\'{e}bec (QC), Canada G1V 0A6}
\email{ransford@mat.ulaval.ca}
\thanks{The last-named author was supported in part by grants from NSERC and the Canada Research Chairs program. The first-named author was supported in part by an Arthur J. O'Leary Faculty Recognition Award from the College of the Holy Cross. Most of the results described in this paper were obtained while the authors were participants in a ``Research In Teams'' group at the Banff International Research Station (BIRS).  We are grateful to BIRS for their hospitality.}



\date{\today}


\begin{abstract}
We give conditions characterizing holomorphic and meromorphic functions in the unit disk of the complex plane in terms of certain weak forms of the maximum principle. Our work is directly inspired by recent results of John Wermer, and by the theory of the projective hull of a compact subset of complex projective space developed by Reese Harvey and Blaine Lawson.
\end{abstract}

\maketitle

\section{Introduction} \label{intro}

\noindent For a compact subset $K$ of complex Euclidean space $\mathbb{C}^n$, we let $C(K)$ be the Banach space of continuous complex-valued functions on $K$ with norm \[ \|f\|_{K} = \max\{ |f(\zeta)| : \zeta \in K \}. \]
The classical maximum principle states that if $\Omega$ is a bounded connected open subset of $\mathbb{C}^n$, and $f$ is holomorphic and non-constant in $\Omega$ and continuous on $\overline{\Omega}$, then for all $z \in \Omega$
\[  |f(z)| < \|f\|_{\partial \Omega}. \]
Our goal in this paper is to explore converses to this basic property of holomorphic functions, i.e., to establish characterizations of holomorphic and meromorphic functions in terms of the existence of certain weak forms of the maximum principle.

One such result, established by Walter Rudin in 1953 \cite{R1}, goes as follows: let $\mathbb{D}$ be the open unit disk in the complex plane, and let $\mathbb{T}$ be its boundary.  Let $A$ be an algebra of continuous complex-valued functions on the closed unit disk $\overline{\mathbb{D}}$ such that \begin{enumerate}[i.]
       \item for each $f \in A$ and $z \in \mathbb{D}$, we have $|f(z)| \leq \|f\|_{\mathbb{T}}$, and
       \item $A$ contains the disk algebra $A(\mathbb{D})$ of all functions holomorphic in $\mathbb{D}$ and continuous in $\overline{\mathbb{D}}$,
     \end{enumerate}
then $A = A(\mathbb{D})$.  In fact, this is a consequence of more general theorems that Rudin proved in \cite{R1} concerning maximum modulus algebras on plane sets.

Rudin  (\cite{R2}, Theorem 12.12) later relaxed the requirement that the set of functions in question be an algebra. He showed that if $V$ is a vector subspace of the space $C(\overline{\mathbb{D}})$ of continuous complex-valued functions on $\overline{\mathbb{D}}$ such that
\begin{enumerate}[i.]
       \item for each $f \in V$ and $z \in \mathbb{D}$, we have $|f(z)| \leq \|f\|_{\mathbb{T}}$;
       \item $V$ contains the constant function $1$; and
       \item for each $f \in V$, the function $g(z) = zf(z)$ also belongs to $V$,
     \end{enumerate}
then $V = A(\mathbb{D})$.

Note that the hypothesis (i) in both of these results is a weak form of the maximum principle: we do not demand that the maximum of $|f|$ is attained \emph{only} on $\mathbb{T}$.

Rudin's later theorem can be rephrased as a statement about modules over the disk algebra: choose $\phi  \in C(\overline{\mathbb{D}})$, and consider the space
\begin{equation} \mathfrak{M} = \{ a + b\phi : a,b \in A(\mathbb{D})  \}.  \label{eq:module} \end{equation}  If for each $f \in \mathfrak{M}$ and $z \in \mathbb{D}$, \begin{equation} |f(z)| \leq \|f\|_{\mathbb{T}} \label{eq:weakmax} \end{equation} then $\mathfrak{M} = A(\mathbb{D})$, i.e.,
$\phi \in A(\mathbb{D})$.  These two formulations are easily seen to be equivalent: first, assuming that (\ref{eq:weakmax}) holds, then $\mathfrak{M}$ is a vector subspace of $C(\overline{\mathbb{D}})$ satisfying the hypotheses (i), (ii) and (iii) of Rudin's result, and so $\mathfrak{M} = A(\mathbb{D})$.   Conversely, given any vector space $V$ satisfying (i), (ii) and (iii), choose any $\phi \in V$.  Then by assumptions (ii) and (iii), $V$ contains all functions of the form $ a + b\phi$ where $a,b$ are  polynomials.  Since polynomials are dense in $A(\mathbb{D})$, assumption (i) implies that (\ref{eq:weakmax}) holds for the space $\mathfrak{M}$ defined in (\ref{eq:module}), and so $\phi$ is holomorphic in $\mathbb{D}$.  Since $\phi \in V$ was arbitrary, $V = A(\mathbb{D})$.

Note that $\mathfrak{M}$ is a module over the disk algebra generated by the constant function 1 and $\phi$.  In fact, the same result holds for any subset $\mathfrak{M}^{\prime}$ of $C(\overline{\mathbb{D}})$ that forms a module over the disk algebra (with the usual multiplication of functions) and contains 1: if we assume (\ref{eq:weakmax}) holds for all $f \in \mathfrak{M}^{\prime}$ then $\mathfrak{M}^{\prime} = A(\mathbb{D})$. To see this, choose any $\phi$ in $\mathfrak{M}^{\prime}$, form the submodule $\mathfrak{M} \subset \mathfrak{M}^{\prime}$ as in (\ref{eq:module}), and conclude that $\phi \in A(\mathbb{D})$.  The same remark obtains for the result we discuss next.

Recently, John Wermer (\cite{W2}, Theorem 2) gave the following generalization of Rudin's Theorem:

\begin{theorem} \label{wermer1} Let $\mathfrak{M}$ be as in (\ref{eq:module}), and assume that for each $z \in \mathbb{D}$ there exists a constant $C_{z}$ so that
\begin{equation} |f(z)| \leq C_{z}\|f \|_{\mathbb{T}} \label{eq:wermermax} \end{equation}
for all $f \in \mathfrak{M}$.  If the restriction of $\phi$ to $\mathbb{T}$ is real-analytic, then $\phi$ is holomorphic in $\mathbb{D}$.
 \end{theorem}
In section \ref{modules} we will strengthen Wermer's result (see Theorem \ref{pointest}, Corollary \ref{cor}, and Remark \ref{discrete}) to give a characterization of functions meromorphic in the disk, while also relaxing the requirement that $\phi$ be real-analytic on the circle.

Wermer's hypothesis (\ref{eq:wermermax}) was inspired by the concept of the \emph{projective hull} $\hat{K}$ of a compact subset $K$ of projective space $\mathbb{P}^n$ introduced by Harvey and Lawson in \cite{HL}.  For a compact subset $K$ of an affine coordinate chart $\mathbb{C}^n$ the intersection of $\hat{K}$ with $\mathbb{C}^n$ is the set of points $z \in \mathbb{C}^n$ such that there exists a constant $C_{z}$ with
\begin{equation} |p(z)| \leq C_{z}^{ \deg{p}}\|p\|_{K} \label{eq:projhull} \end{equation}
for all polynomials $p$ (see \cite{HL}, section 6).  For compact subsets $K$ of $\mathbb{C}^n$, the projective hull is thus a generalization of the polynomially convex hull $\hat{K}_{P}$, defined as the set of points $z \in \mathbb{C}^n$ for which (\ref{eq:projhull}) holds with $C_{z} = 1$.

Harvey and Lawson were in turn motivated by the desire to find an appropriate generalization for curves in projective space of a celebrated theorem of Wermer \cite{W1}: if $\gamma \subset \mathbb{C}^n$ is a closed real-analytic curve, and $\hat{\gamma}_{P} \neq \gamma$, then $\hat{\gamma}_{P} \setminus \gamma$ is a one-dimensional complex-analytic variety.  This result was later extended by various authors to smooth and even merely rectifiable curves.  However, the analogue of Wermer's theorem for projective hulls may fail spectacularly if $\gamma$ is smooth but not real-analytic: Diederich and Fornaess \cite{DF} constructed a $C^{\infty}$ curve $\gamma \subset \mathbb{C}^2$ such that $\gamma$ is not pluripolar, and (see section 3) this is equivalent to the statement that $\hat{\gamma} = \mathbb{P}^2$.

Wermer also considered characterizations of functions meromorphic in the punctured disk $\mathbb{D}^* := \mathbb{D} \setminus \{0\}$ in terms of projective hulls of graphs.  If $\phi \in C(\mathbb{D}^* \cup \mathbb{T})$ we set
\[ \Sigma := \{ (z,\phi(z)) : z \in \mathbb{D}^* \} \]
and \[ \gamma := \{ (z, \phi(z)) : z \in \mathbb{T} \}. \] Wermer showed (\cite{W2}, Proposition 3)
\begin{theorem} \label{wermer2} If $\phi \in C(\mathbb{D}^* \cup \mathbb{T})$ is meromorphic in $\mathbb{D}$, then $\Sigma \subset \hat{\gamma}$. \end{theorem}
Note that the condition $\Sigma \subset \hat{\gamma}$ is equivalent to the statement that for each $z \in \mathbb{D}^*$ there exists a constant $C_{z}$ such that \begin{equation} |p(z,\phi(z))| \leq C_{z}^{\deg p} \|p\|_{\gamma} \label{eq:hatgamma} \end{equation}
for all polynomials $p$ in two variables, while the hypothesis (\ref{eq:wermermax}) in Theorem \ref{wermer1} is equivalent to the existence of a constant $C_{z}$ such that
\[ |p(z,\phi(z))| \leq C_{z}\|p\|_{\gamma} \]
for all polynomials $p(z,w)$ of degree at most one in $w$ and of arbitrary degree in $z$.

In section \ref{graphs} we will consider the question of a converse to Theorem \ref{wermer2}.  The existence (noted above) of smooth curves with projective hull equal to $\mathbb{P}^2$ makes it necessary to place some restriction on the curve $\gamma$.   Wermer in \cite{W2} posed the question: {\em Suppose $\phi \in C(\mathbb{D}^* \cup \mathbb{T})$ and $\phi$ is real-analytic on $\mathbb{T}$.  If $\Sigma \subset \hat{\gamma}$, is $\phi$ in fact meromorphic in $\mathbb{D}$?}
  Under a much stronger assumption on $\phi$ (namely, that $\phi$ is in fact real-analytic on all of $\mathbb{C}^2$), Wermer showed that the answer is yes (in this case the conclusion is that $\phi$ is in fact an entire function).

 We will show in section 3 (Theorem \ref{holomorphic}) that if $\phi \in C(\mathbb{D}^* \cup \mathbb{T})$ and $\gamma$ is pluripolar, then using a deep result of Shcherbina it is easy to conclude that $\Sigma \subset \hat{\gamma}$ implies that $\phi$ is holomorphic in $\mathbb{D}^*$.  In order to show that $\phi$ has at most a pole at the origin, we will make a stronger assumption on the projective hull (Theorem \ref{meromorphic}): we assume that $\Sigma = \hat{\gamma} \cap (\mathbb{D}^* \times \mathbb{C})$.  We do not know if this stronger assumption is necessary.

 Finally, we note that Josip Globevnik has recently given in \cite{G} a characterization of meromorphic functions that bears some resemblance to the Wermer/Rudin results but postulates a condition on winding numbers: $\phi \in C(\mathbb{T})$ extends meromorphically to $\mathbb{D}$ iff there exists $N \in \mathbb{N} \cup \{0\}$ such that for all $f$ of the form $a + b\phi$ with $a,b \in A(\mathbb{D})$ the change in the argument of $f$ around $\mathbb{T}$ is bounded by $-2\pi N$.  We are indebted to Dmitry Khavinson for pointing out Globevnik's work to us.

\section{Modules over the disk algebra} \label{modules}

\noindent We will denote by $H^{p}$ the Hardy space consisting  of functions holomorphic in the disk and satisfying (if $0 < p < \infty$)
\[ \|f\|_{p}^{p} := \sup_{r < 1}\frac{1}{2\pi} \int_{0}^{2\pi} |f(re^{i\theta})|^{p} d\theta  < \infty,  \]
while $H^{\infty}$ will denote the space of functions bounded and holomorphic in the disk, with $\|f\|_{\infty} = \sup\{ |f(z)|: |z| < 1 \}$.  For $\zeta \in\mathbb{T}$ and $\alpha > 1$ we let $\Gamma_{\alpha}(\zeta)$ be the nontangential approach region \[ \Gamma_{\alpha}(\zeta) := \{ z \in \mathbb{D} : |z - \zeta| < \alpha(1 -  |z|)\}. \]  Then for a function $f \in H^{p}$, $0 < p \leq \infty$, the nontangential limit $f^{*}(\zeta) = \lim_{z \in \Gamma_{\alpha}(\zeta) \rightarrow \zeta} f(z)$ exists for almost every $\zeta \in \mathbb{T}$ and every $\alpha > 1$.  Moreover $f^{*} \in L^{p}(\mathbb{T})$ with respect to normalized Lebesgue measure $d\theta/2\pi$ on the circle and $\|f\|_{p} = \|f^*\|_{L^{p}(\mathbb{T})}$, for $0 < p \leq \infty$.

We first give a condition equivalent to the existence of the inequality (\ref{eq:wermermax}) at a single point, assuming only that $\phi \in C(\mathbb{T})$.

\begin{theorem} \label{pointest} If $\phi \in C(\mathbb{T})$, $z \in \mathbb{D}$, and $\lambda \in \mathbb{C}$, then the following are equivalent:
\begin{enumerate}
  \item There exists a constant $C$ such that $|a(z) + b(z)\lambda| \leq C\|a + b\phi\|_{\mathbb{T}}$ for all $a,b \in A(\mathbb{D})$;
  \item There exist bounded holomorphic functions $f,g$ in $\mathbb{D}$ such that $h = f/g$ satisfies $h(z) = \lambda$ and $h^{*} = \phi$ a.e. on $\mathbb{T}$.
\end{enumerate}
\end{theorem}

\begin{proof} First, suppose that (1) holds. Since composition with a fixed element of the holomorphic automorphism group of the disk preserves $A(\mathbb{D})$ and is a norm-preserving map of $C(\mathbb{T})$ to itself, we may assume that $z = 0$.  Let \[ \mathfrak{M} = \{ a + b\phi : a,b \in A(\mathbb{D})\}. \] Restricting functions in $\mathfrak{M}$ to the circle, we may consider $\mathfrak{M}$ as a subspace of $C(\mathbb{T})$.  By (1) the linear functional given by $a + b\phi \mapsto a(0) + b(0)\lambda$ is bounded on $\mathfrak{M}$, and therefore by the Hahn-Banach theorem extends to a continuous linear functional on $C(\mathbb{T})$.  By the Riesz representation theorem there exists a complex measure $\mu$ supported on $\mathbb{T}$ representing this linear functional, i.e.,
\begin{equation} \int_{\mathbb{T}} (a + b\phi)\, d\mu = a(0) + b(0)\lambda \label{eq:LF} \end{equation}
for all $a,b \in A(\mathbb{D})$.  Taking  $b \equiv 0$ and $a(z) = z^n$ in (\ref{eq:LF}) we find that for all $n \geq 1$,
\[ \int_{\mathbb{T}} z^{n}  d\mu(z)  = 0, \]
while for $n = 0$ we have $\int_{\mathbb{T}} d\mu = 1$.  Thus the measure $\nu$ defined by $d\nu = d\mu - d\theta/2\pi$ has the property that
\[ \int_{\mathbb{T}} z^{n} d\nu(z) = 0 \]
for all $n \geq 0$.  By the F. and M. Riesz theorem, there exists a function $k \in H^{1}(\mathbb{D})$ with $k(0) = 0$ such that $d\nu = k^{*}d\theta/2\pi$.  Solving for $d\mu$ and replacing $k$ by $k + 1$, we find
\begin{equation} d\mu = k^*d\theta/2\pi \label{eq:mu} \end{equation} with $k \in H^{1}$ and $k(0) = 1$.

Next, in (\ref{eq:LF}) we take $a \equiv 0$ and $b(z) = z^n$ to obtain for $n \geq 1$
\[ \int_{\mathbb{T}} z^n \phi(z) \, d\mu = 0, \]
while for $n = 0$ we have $\int_{\mathbb{T}} \phi \, d\mu = \lambda$.  Arguing as above with $\mu$ replaced by $\phi \, d\mu$ we obtain a function $\ell \in H^{1}$ with $\ell(0) = \lambda$ and \begin{equation} \phi \, d\mu = \ell^* d\theta/2\pi. \label{eq:phimu} \end{equation}
Multiplying both sides of (\ref{eq:mu}) by $\phi$ and comparing with (\ref{eq:phimu}) we see that $\phi k^* = \ell^{*}$ a.e. on $\mathbb{T}$.  Since $k(0) = 1$, $k$ is not identically zero, so $k^*$ is non-zero a.e. on $\mathbb{T}$. Thus if we set $h = \ell/k$,
  \[ \phi = h^{*}  \]
  a.e. on $\mathbb{T}$, and $h(0) = \ell(0)/k(0) = \lambda$.
  Since every function in $H^{1}$ can be represented as the quotient of two functions holomorphic and bounded in $\mathbb{D}$ (this is true more generally for functions in the Nevanlinna class - see \cite{Du}, Theorem 2.1), the same is true for $h$.  Thus we may write $h = f/g$ with $f,g$ bounded, establishing (2).

Now suppose that (2) holds, so that $h = f/g$ with $f,g$ bounded and holomorphic in $\mathbb{D}$, $h(z) = \lambda$, and $h^{*} = \phi$ a.e in $\mathbb{T}$.  Note $g(z) \neq 0$.  Given $a,b \in A(\mathbb{D})$, we have
\begin{eqnarray*}  |a(z) + b(z)\lambda| &  =  & |a(z) + b(z)f(z)/g(z)|  \\ &  = & \frac{1}{|g(z)|} \cdot |g(z)a(z) + b(z)f(z)| \\
& = & \frac{1}{|g(z)|} \|ga + bf\|_{\infty} \\ & \leq & \frac{1}{|g(z)|} \cdot \|g^{*}\|_{L^\infty(\mathbb{T})} \cdot \|a + bf^{*}/g^{*}\|_{L^\infty(\mathbb{T})} \\ & & \\ & = & C\|a + b\phi\|_{\mathbb{T}}\end{eqnarray*}
taking $C = C_{z} = \|g^{*}\|_{L^{\infty}(\mathbb{T})}/|g(z)|$.  Here we have used the facts that $ga + bf \in H^{\infty}$ and that $f^{*}/g^{*}$ agrees a.e. on $\mathbb{T}$ with the continuous function $\phi$.  This shows that (1) holds, and completes the proof of Theorem \ref{pointest}.
\end{proof}

\begin{remark} Condition (2) can be interpreted as saying that there is a meromorphic function $h$ whose graph over the disk passes through the point $(z,\lambda) \in \mathbb{C}^2$ and whose ``boundary'' lies on the graph $\gamma$ of $\phi$ over $\mathbb{T}$, in the sense that $(e^{i\theta},h^*(e^{i\theta})) \in \gamma$ for almost all $\theta$.
\end{remark}

\begin{corollary} \label{cor} Let $\Omega \subset \mathbb{D}$ be open, and assume $\phi \in C(\Omega \cup \mathbb{T})$.  Suppose that for each $z \in \Omega$ there exists a constant $C_{z}$ such that for all $a,b \in A(\mathbb{D})$,
\[ |a(z) + b(z)\phi(z)| \leq C_{z}\|a +b\phi\|_{\mathbb{T}}. \]
 Then there exist $f,g \in H^{\infty}$ such that $h = f/g$ satisfies $h^{*} = \phi$ a.e. on $\mathbb{T}$ and $h = \phi$ on $\Omega$.  In particular, $\phi$ is holomorphic in $\Omega$ and extends to be meromorphic in $\mathbb{D}$.
\end{corollary}

\begin{proof} By Theorem \ref{pointest}, for each fixed $z \in \Omega$, there exist $f,g \in H^{\infty}$ such that if $h = f/g$, we have $h^* = \phi$ a.e. on $\mathbb{T}$ and $h(z) = \phi(z)$.  The functions $f,g$ and $h$ depend \emph{a priori} on the point $z$, but note that the condition $h^* = \phi$ a.e. uniquely determines $h$: if $h_{1}, h_{2}$ are both quotients of $H^{\infty}$ functions and $h_{1}^{*} = h_{2}^{*}$ a.e on $\mathbb{T}$, then $h_{1} = h_{2}$.     We therefore obtain a single function $h$, and (we may choose) a single pair of bounded functions $f,g$, so that $h = f/g$ on $\mathbb{D}$ and $h(z) = \phi(z)$ for all $z \in \Omega$.  Thus both $\phi$ and $h$ are in fact holomorphic in $\Omega$, and $\phi$ extends meromorphically (as $h$) to $\mathbb{D}$.
\end{proof}

\begin{remark} \label{discrete} Let $F$ be a discrete subset of $\mathbb{D}$, and $\Omega = \mathbb{D} \setminus F$.  Corollary \ref{cor} then gives a characterization of functions meromorphic in $\mathbb{D}$ whose poles are contained in the set $F$.  In particular, taking $F$ to be the empty set, we obtain a generalization of Wermer's result Theorem \ref{wermer1} in which the function $\phi$ is not required to be real-analytic on $\mathbb{T}$.
\end{remark}

\section{Projective hulls of Graphs} \label{graphs}

\noindent In this section we consider the possibility of a converse to Theorem \ref{wermer2}.  Our notation is as in the Introduction: $\mathbb{D}^* = \mathbb{D} \setminus \{0\}$, we fix a function $\phi$ continuous on $\mathbb{D}^{*} \cup \mathbb{T}$, and $\Sigma, \gamma$ are the graphs of $\phi$ over $\mathbb{D}^*$ and $\mathbb{T}$, respectively.

We begin with some necessary preliminaries from pluripotential theory.  Given a domain $\Omega \subset \mathbb{C}^n$, we let $\mbox{PSH}(\Omega)$ be the set of functions plurisubharmonic in $\Omega$, i.e., the set of upper semi-continuous functions $u : \Omega \rightarrow [-\infty, \infty)$ which are subharmonic when restricted to (each component of) $L \cap \Omega$, for any complex line $L$ in $\mathbb{C}^n$.   If $g$ is a function holomorphic in $\Omega$, then $\log|g| \in \mbox{PSH}(\Omega)$.  We will use the fact that the maximum of a finite subset of $\mbox{PSH}(\Omega)$ is also plurisubharmonic in $\Omega$.

For a compact set $K \subset \mathbb{C}^n$, set \[ V_{K}(z) = \sup \{ \frac{1}{\deg(p)}\log|p(z)|: p \mbox{ is a polynomial with } \|p\|_{K} \leq 1 \}. \]
It is clear from the characterization (\ref{eq:projhull}) of the projective hull that
\[ \hat{K} \cap \mathbb{C}^n = \{ z: V_{K}(z) < \infty \}, \] and that if $z \in \hat{K} \cap \mathbb{C}^n$, then $V_K(z)=\log C_z$, where $C_z$ is the least constant such that (\ref{eq:projhull}) holds.  Each of the functions $u = \log|p|/\deg{p}$ appearing in the definition of $V_{K}$ satisfies $u \leq 0$ on $K$ and is easily seen to be a member of the Lelong class $L(\mathbb{C}^n)$ of plurisubharmonic functions of logarithmic growth defined by
\begin{equation} L(\mathbb{C}^n) = \{ u \in \mbox{PSH}(\mathbb{C}^n):  u(z) \leq \log(1 + |z|) + c, \forall z \in \mathbb{C}^n \} \label{eq:lelong} \end{equation}
where in $(\ref{eq:lelong})$ the constant $c$ is allowed to depend on the function $u$.  For $K$ compact we may also obtain (see \cite{Si} or \cite{Z}) the extremal function $V_{K}$ by taking the supremum over the larger class:
\[ V_{K}(z) = \sup \{ u(z): u \in L(\mathbb{C}^n), u \leq 0 \mbox{ on } K \}. \]
Note that $V_{K} \geq 0$, and thus, by replacing each function $\log|p|/\deg{p}$ by $\max\{ 0, \log|p|/\deg{p} \}$ we may realize $V_{K}(z)$ as an upper envelope of \emph{continuous} functions in the Lelong class.

A set $E \subset \mathbb{C}^n$ is said to be {\em pluripolar} if there exists a function $u$ which is plurisubharmonic  on $\mathbb{C}^n$ and not identically equal to $-\infty$, such that $E \subset \{ z \in \mathbb{C}^n: u(z) = -\infty \}$.  By a result of Josefson \cite{J}, this is equivalent to the statement that $E$ is \emph{locally pluripolar}, i.e.,  for each $z_{0} \in E$, there is an open neighborhood $U$ of $z_{0}$ and $u \in \mbox{PSH}(U)$ with $E \cap U \subset \{ z \in U: u(z) = -\infty \}$.  It follows that any complex analytic subvariety of an open subset of $\mathbb{C}^n$ is pluripolar.

 Using a result of Guedj and Zeriahi \cite{GZ}, Harvey and Lawson prove (\cite{HL}, Corollary 4.4) the following fundamental relationship between  pluripolarity and the projective hull:
\begin{proposition} \label{hulls} For a compact set $K \subset \mathbb{P}^n$, \[  K \mbox{ is pluripolar } \Leftrightarrow \hat{K} \mbox{ is pluripolar } \Leftrightarrow \hat{K} \neq \mathbb{P}^n. \]
\end{proposition}

If $\gamma \subset \mathbb{C}^n$ is a real-analytic curve, then $\gamma$ is locally contained in a complex analytic variety, and so is pluripolar.
However there are (as we noted in the introduction) smooth curves $\gamma \subset \mathbb{C}^2$ with $\gamma$ not pluripolar, implying $\hat{\gamma} = \mathbb{P}^2$.  (The paper \cite{CLP} explores conditions under which a smooth curve is pluripolar.) It is therefore natural, in considering the converse of Theorem \ref{wermer2}, to impose the condition that $\gamma$ is pluripolar.  With this assumption we have:

\begin{proposition} \label{holomorphic} Suppose $\phi \in C(\mathbb{D}^* \cup \mathbb{T})$ and that $\gamma$ is pluripolar.  If $\Sigma \subset \hat{\gamma}$, then $\phi$ is holomorphic in $\mathbb{D}^*$.  \end{proposition}

\begin{proof} A result of Shcherbina \cite{Sh} states that if $\Omega$ is a domain in $\mathbb{C}^n$, and $f \in C(\Omega)$, then the graph of $f$ over $\Omega$ is pluripolar if and only if $f$ is holomorphic in $\Omega$.  Our assumption on $\gamma$ implies that $\hat{\gamma}$ is a proper subset of ${\bf C}^2$, and so $\gamma$, and therefore $\hat{\gamma}$ is pluripolar. Since
$\Sigma$ is contained in $\hat{\gamma}$, it too is pluripolar. By Shcherbina's result, $\phi$ is holomorphic in $\mathbb{D}^*$. \end{proof}

Given Proposition \ref{holomorphic} and our previous remarks, it is natural to make the following conjecture (cf.
Wermer's question mentioned in the Introduction):

\begin{conjecture}
\label{conj}
If $\gamma$ is pluripolar and $\Sigma \subset \hat{\gamma}$, then $\phi$ is meromorphic in $\mathbb{D}$.
\end{conjecture}

As a first step we prove that if the extremal function $V_{\gamma}$ is harmonic on $\Sigma$, then $\phi$ is in fact meromorphic.

\begin{lemma} \label{harmonic} Under the hypotheses of Proposition \ref{holomorphic}, if $z \mapsto V_{\gamma}(z,\phi(z))$ is harmonic on $\mathbb{D}^*$, then $\phi$ is meromorphic in $\mathbb{D}$. \end{lemma}

\begin{proof} Let $h(z) = -V_{\gamma}(z,\phi(z))$.  Since $V_{\gamma} \geq 0$, $h$ is harmonic and bounded above on $\mathbb{D}^*$, and thus extends to be subharmonic on $\mathbb{D}$ (the union of $\mathbb{D}^*$ with the polar set $\{0\}$) by defining $h(0) = \limsup_{z \rightarrow 0} h(z)$.  Then $\triangle u = \alpha \delta_{0}$ for some constant $\alpha$, where $\delta_{0}$ is the unit point mass at the origin, so
\[ \beta(z) := h(z) - \alpha \frac{\log |z|}{2 \pi} \]
is in fact harmonic on $\mathbb{D}$, and therefore is bounded below near the origin.  It follows that there exists a constant $m \geq 0$ such that
\begin{equation} V_{\gamma}(z,\phi(z)) \leq -m \log |z| \label{eq:log} \end{equation}
Recall that $V_{\gamma}(z,\phi(z)) = \log C_{z}$, where $C_{z}$ is the least constant such that $|p(z)| \leq C_{z}^{\deg{p}} \|p\|_{\gamma}$ for all polynomials $p$.  Taking $p(z,w) = w$ we obtain for all $z$ near the origin
\[ |\phi(z)| \leq C_{z} \|\phi\|_{\mathbb{T}} = \exp(V_{\gamma}(z,\phi(z))) \|\phi\|_{\mathbb{T}}  \leq C|z|^{-m} \]
implying that $\phi$ has a pole of order at most $m$ at $z = 0$.
\end{proof}

Note that while the assumption $\Sigma \subset \hat{\gamma}$ implies that $V_{\gamma}$ is finite on $\Sigma$, \emph{a priori} we have no reason to believe that $V_{\gamma}$ has any regularity, e.g., is even locally bounded, on $\Sigma$.  However, by making a stronger assumption on $\hat{\gamma}$ we can prove that $V_{\gamma}$ is harmonic and thereby establish the following special case of Conjecture \ref{conj}:

\begin{theorem} \label{meromorphic} If $\phi \in C(\mathbb{D}^* \cup \mathbb{T})$ and $\Sigma = \hat{\gamma} \cap (\mathbb{D}^* \times \mathbb{C})$, then $\phi$ is meromorphic in $\mathbb{D}$.  \end{theorem}

The proof of Theorem \ref{meromorphic} follows closely the reasoning in the proof of Theorem 7.3 of \cite{HL}, which is itself an adaptation of arguments of Sadullaev in \cite{Sa}. It is based upon the following lemma.

\begin{lemma} \label{modify} Under the hypotheses of Theorem \ref{meromorphic}, assume $u \in L(\mathbb{C}^2) \cap C(\mathbb{C}^2)$ with $u \leq 0$ on $\gamma$.  Fix an open disk $D$ with $\overline{D} \subset \mathbb{D}^{*}$, and let $v$ be the harmonic extension to $D$ of the restriction of $u(z,\phi(z))$ to $\partial D$.  Then for every $\epsilon > 0$ there exists $u_{\epsilon} \in L(\mathbb{C}^2) \cap C(\mathbb{C}^2)$ with $u_{\epsilon} \leq 0$ on $\gamma$ and
\begin{equation} u_{\epsilon}(z,\phi(z)) \geq \max\{ u(z,\phi(z)), v(z) - \epsilon \} \label{eq:harmonic} \end{equation}
for all $z \in \overline{D}$.
\end{lemma}

\begin{proof} For $z \in \overline D$ and $w \in \mathbb{C}$ set $\tilde{v}(z,w) = v(z)$.  By the continuity of $u$ and $v$ we may choose $\rho > 0$ sufficiently small so that \begin{equation} \tilde{v} - \epsilon < u \mbox{ on } X := \{ (z,w): z \in \partial D, |w - \phi(z)| \leq \rho\}. \label{eq:x} \end{equation}  Since $\Sigma = \hat{\gamma} \cap (\mathbb{D}^* \times \mathbb{C})$, and $V_{\gamma} = \infty$ on $\mathbb{C}^2 \setminus \hat{\gamma}$, for each point $q \in (\mathbb{D}^* \times \mathbb{C}) \setminus \Sigma$,  we may choose a function $U_{q} \in L(\mathbb{C}^2) \cap C(\mathbb{C}^2)$ with $U_{q} \leq 0$ on $\gamma$ and $U_{q}(q) > \|v\|_{\overline{D}}$.  Thus by compactness of $Y := \{ (z,w): z \in \overline{D}, |w - \phi(z)| = \rho\}$, there exist finitely many functions $U_{1}, \ldots ,U_{N} \in L(\mathbb{C}^2) \cap C(\mathbb{C}^2)$ so that \begin{equation} U := \max\{ U_{1}, \ldots ,U_{N} \} > \|v\|_{\overline{D}} = \|\tilde{v}\|_{Y} \mbox{ on } Y\label{eq:y} \end{equation} and $U \leq 0$ on $\gamma$.  Then $U \in L(\mathbb{C}^2) \cap C(\mathbb{C}^2)$.  Set $\Omega = \{ (z,w): z \in D, |w - \phi(z)| < \rho \}$, and define
\[ u_{\epsilon} = \left\{ \begin{array}{ll}
                          \max\{ U, u, \tilde{v} - \epsilon \}  & \mbox{ on } \Omega \\
                            \max\{ U, u \} & \mbox{ on } \mathbb{C}^2 \setminus \Omega.
                          \end{array} \right. \]
                          Since $\partial \Omega = X \cup Y$, equations (\ref{eq:x}) and (\ref{eq:y}) imply that $u_{\epsilon}$ is continuous.  Moreover, $U$ and $u$ are plurisubharmonic on $\mathbb{C}^2$ while $\tilde{v}$ is plurisubharmonic on $\Omega$, so by a standard ``gluing'' result (see \cite{Kl}, Corollary 2.9.15) $u_{\epsilon}$ is plurisubharmonic on $\mathbb{C}^2$.  Since both $U$ and $u$ are of logarithmic growth, and less than or equal to zero on $\gamma$, the same is true of $u_{\epsilon}$ (note $\gamma \subset \mathbb{C}^2 \setminus \Omega$).  Finally, $\overline{D} \times \mathbb{C} \subset \overline{\Omega}$, so (\ref{eq:harmonic}) clearly holds.
 \end{proof}

\begin{proof}[Proof of Theorem \ref{meromorphic}]
By Lemma \ref{harmonic}, to show that $\phi$ is meromorphic, it suffices to show that $V_{\gamma}(z,\phi(z))$ is harmonic in $\mathbb{D}^*$.  To this end, fix an arbitrary disk $D$ with $\overline{D} \subset \mathbb{D}^*$ and a point $z_{0} \in D$.  We may choose a sequence
$\{ u_m \} \subset L(\mathbb{C}^2) \cap C(\mathbb{C}^2)$ with $u_m \leq 0$ on $\gamma$
such that $\lim_{m \rightarrow \infty} u_{m}(z_{0}) = V_\gamma(z_0,\phi(z_0))$. Replacing $u_{m}$ with $\max\{ u_{1}, \ldots , u_{m}\}$, we may assume that the sequence $u_{m}$ is increasing.  Let $v_{m}$ be the harmonic extension to $D$ of the restriction of $u_m(z,\phi(z))$ to $\partial D$.  Note that since $\{u_{m} \}$ is increasing, the sequence $\{v_{m}\}$ is also increasing.  Taking $\epsilon = 1/m$, Lemma \ref{modify} gives a function $\tilde{u}_m = u_{m,1/m}$ with
\begin{equation} u_{m}(z,\phi(z)) - 1/m \leq v_{m}(z) - 1/m \leq \tilde{u}_{m}(z,\phi(z)) \leq V_{\gamma}(z,\phi(z)) \label{eq:harm} \end{equation}
for all $z \in \overline{D}$, the rightmost inequality following from the fact that $\tilde{u}_{m} \in L(\mathbb{C}^2)$ and $\tilde{u}_{m} \leq 0$ on $\gamma$.  By (\ref{eq:harm}) and our choice of the sequence $u_{m}$, \[ \lim_{m \rightarrow \infty} v_{m}(z_{0}) = V_{\gamma}(z_{0},\phi(z_{0})) < \infty. \]  Harnack's Theorem (\cite{Ra}, Theorem 1.3.9) then implies that the sequence $v_{m}$ converges locally uniformly to a function $s$ harmonic in $D$; by (\ref{eq:harm}) we have $s(z) \leq V_{\gamma}(z,\phi(z))$ for all $z \in D$.  Now choose any point $z_{1} \in D, z_{1} \neq z_{0}$, and repeat the above argument, beginning with an increasing sequence $\{U_{m}\} \subset L(\mathbb{C}^2) \cap C(\mathbb{C}^2)$ such that \begin{equation} \lim_{m \rightarrow \infty} U_{m}(z_{1}) = V_{\gamma}(z_{1},\phi(z_{1})). \label{eq:zlim} \end{equation} We may replace $U_{m}$ by $\max\{U_{m},u_{m}\}$ and retain (\ref{eq:zlim}), since $U_{m}$ and $u_{m}$ are both bounded above by $V_{\gamma}$; thus we may assume that $U_{m} \geq u_{m}$.  Arguing as above, we obtain a function $t$ harmonic in $D$ with $t(z_{1}) = V_{\gamma}(z_{1},\phi(z_{1}))$ and $t(z) \leq V_{\gamma}(z,\phi(z))$ for all $z \in D$.  Moreover, since $U_{m} \geq u_{m}$, we have $t \geq s$ on $D$.  But then \[ V_{\gamma}(z_{0},\phi(z_{0})) = s(z_{0}) \leq t(z_{0}) \leq V_{\gamma}(z_{0},\phi(z_{0})), \] implying $s(z_{0}) = t(z_{0})$.  By the maximum principle, $s \equiv t$ on $D$, and so the harmonic function $s(z)$  is equal to $V_{\gamma}(z,\phi(z))$ when $z = z_{0}, z_{1}$.  But $z_{1}$ was an arbitrary point in $D$, and therefore $s(z) = V_{\gamma}(z,\phi(z))$ for all $z \in D$, implying $V_{\gamma}(z,\phi(z))$ is harmonic in $D$.  Since $D$ was arbitrary, $V_{\gamma}(z,\phi(z))$ is harmonic on $\mathbb{D}^*$.  This completes the proof. \end{proof}

\begin{remark} The hypothesis $\Sigma = \hat{\gamma} \cap (\mathbb{D}^* \times \mathbb{C})$ was used only in the proof of Lemma \ref{modify}, to ensure that (locally) points at some fixed small distance from $\Sigma$ are disjoint from $\hat{\gamma}$, and therefore $V_{\gamma} = \infty$ at such points.  We do not know if this assumption is necessary in order to conclude that $\phi$ is meromorphic, but it is in the spirit of certain ``regularity'' hypotheses of $\hat{\gamma}$ appearing in the work of Harvey and Lawson. For example, Harvey and Lawson prove the following analogue of Wermer's theorem on the polynomial hull of a curve in $\mathbb{C}^n$: if a smooth pluripolar curve $\gamma \subset \mathbb{P}^n$ has the property that the closure of $\hat{\gamma}$ has finite two-dimensional Hausdorff measure in a neighborhood of some complex hypersurface, then $\overline{\hat{\gamma}} \setminus \gamma$ is a 1-dimensional complex analytic subvariety of $\mathbb{P}^n \setminus \gamma$ (\cite{HL}, Theorem 12.9).   \end{remark}

\begin{remark} The condition given in Lemma \ref{harmonic} - that $V_{\gamma}$ is harmonic on $\Sigma$ - is sufficient to show that $\phi$ is meromorphic in $\mathbb{D}^*$, but may not be necessary.  Note that in the proof of Lemma \ref{harmonic} we did not use the full strength of the assumption $\Sigma \subset \hat{\gamma}$, which is equivalent to (\ref{eq:hatgamma}).  Instead, we only used (\ref{eq:hatgamma}) for the particular polynomial $p(z,w) = w$, and so it is possible that the conclusion of Theorem \ref{meromorphic} obtains in cases where $V_{\gamma}$ is not harmonic.
\end{remark}

\end{document}